\documentclass{article}

\usepackage{amsmath}
\usepackage{amssymb}
\usepackage{amsfonts}
\usepackage{amsthm}
\usepackage{color}
\usepackage{graphicx}
\usepackage{hyperref}
\usepackage{ulem}
\usepackage{tikz}
\usetikzlibrary{calc}
\usepackage{mathabx}
\theoremstyle{plain}
\newtheorem{thm}{Theorem}[section]
\newtheorem{lem}[thm]{Lemma}
\newtheorem{prop}[thm]{Proposition}

\newtheorem{cor}[thm]{Corollory}
\theoremstyle{definition}

\theoremstyle{remark}

\begin{document}

\title {\bf An alternative proof of the Dirichlet prime number theorem}

\author{\it Haifeng Xu\thanks{Project supported by NSFC(Grant No. 11401515), the University Science Research Project of Jiangsu Province (14KJB110027, 12KJB110020) and the Foundation of Yangzhou University 2014CXJ004.}}
\date{\small\today}

\maketitle

\begin{abstract}
Dirichlet's theorem on arithmetic progressions called as Dirichlet prime number theorem is a classical result in number theory. Atle Selberg\cite{Selberg} gave an elementary proof of this theorem. In this article we give an alternative proof of it based on a previous result of us. Also we get an estimation of the prime counting function in the special cases.
\end{abstract}




\section{Definitions and some basic facts recalled}

For the details of the notations, definitions and theorems in this section please refer to the paper \cite{Xu-Zhang-Zhou}.

Let $\{p_1,p_2,\ldots,p_n,\ldots\}$ be the set of all primes, where $p_n$ denotes the $n$-th prime number. $\pi(x)$ is the number of primes which are less than or equal to $x$. $\mathbb{N}$ is the set of positive numbers. Let
\[
M_{p_n}=\mathbb{N}-\{2k,3k,5k,\ldots,p_n k\mid k\in\mathbb{N}\},
\]
and $D_{p_n}$ be the set of difference of two consecutive numbers in $M_{p_n}$, that is,
\[
D_{p_n}=\{d_k\mid d_k=x_{k+1}-x_k, x_i\in M_{p_n}\},
\]
where $x_i$ is the $i$-th number in $M_{p_n}$. We call the minimum subset $\mathcal{P}_{p_n}$ occurs periodically in $D_{p_n}$ as the pattern of $D_{p_n}$. The number of elements in the pattern $\mathcal{P}_{p_n}$ is called the period of $D_{p_n}$. We write it as $T_{p_n}$ and it is
\[
T_{p_n}=(2-1)(3-1)(5-1)\cdots(p_n-1).
\]
Note that it is also the number of elements of $D_{p_n}$.

The length of the pattern $\mathcal{P}_{p_n}$ is defined as the sum of the elements in the pattern. We write it as $L(\mathcal{P}_{p_n})$, and it is equal to $\prod_{i=1}^n p_i$. It is also related to the last element in $M_{p_n}^{(0)}$. Here $M_{p_n}^{(0)}$ denotes the first block of $M_{p_n}$ corresponding to $\mathcal{P}_{p_n}$ (just like $M_7^{(0)}$ in Figure \ref{fig:M7}). In fact, the last element in $M_{p_n}^{(0)}$ is $1+L(\mathcal{P}_{p_n})$.

A useful observation is that the number of the elements in $M_{p_n}^{(0)}$ can be computed as follows:
\[
\prod_{i=1}^{m}p_i\cdot(1-\frac{1}{p_1})(1-\frac{1}{p_2})\cdots(1-\frac{1}{p_m})=\prod_{i=1}^{m}(p_i-1)=T_{p_m}.
\]
It shows the relation with the zeta function $\zeta(s)$ in case of $s=1$.

\begin{thm}\label{thm:mirror-symmetry}
The gap sequence except the last element in the pattern is mirror symmetric.
\end{thm}

\begin{cor}\label{cor:1}
The last element in every pattern $\mathcal{P}_{p_n}$ is $2$.
\end{cor}
\begin{thm}\label{thm:about-skips}
(i) The multiplicities of skips $2$ and $4$ are always same and odd. In fact, it is equal to $(3-2)\cdot(5-2)\cdots(p_k-2)$ for $D_{p_k}$.\\
(ii) The central gap is always $4$, and the multiplicities of all gaps except $2$ and $4$ are even.
\end{thm}

\begin{prop}\label{prop:2}
The numbers in the set
\[
\{q\in M_{p_{n-1}}\mid p_n<q<p_n p_{n+1},\quad q\neq p_n^2\}
\]
are consecutive primes.
\end{prop}

In this paper, we always write $\log x$ to stand for $\ln x$.

\section{Dirichlet's theorem}

We will use the Erd\H{o}s estimate about Chebyshev's inequality to prove Dirichlet's theorem.
\begin{lem}[Chebyshev]
There are two positive constant $c_1\leqslant 1\leqslant c_2$, such that
\[
c_1\frac{x}{\log x}<\pi(x)<c_2\frac{x}{\log x}
\]
holds for all $x\geqslant N$ for some positive integer $N$. If set $N=2$, then we can choose $c_1=\frac{1}{3}\log 2$, and $c_2=6\log 2$.
\end{lem}

Erd\H{o}s \cite{Erdos} proved the following inequality in 1932.
\begin{lem}[Erd\H{o}s]\label{eqn:Erdos}
\[
(\log 2)\cdot\frac{x}{\log x}<\pi(x)<(2\log 2)\cdot\frac{x}{\log x},
\]
for all $x\geqslant 3$.
\end{lem}

The original proof by Dirichlet in 1837 used the characters of finite abelian groups and L-seires. In this section we give an alternative method to prove the theorem of Dirichlet.

\begin{thm}[Dirichlet, 1837 \cite{Yessenov,Dirichlet-theorem}]
For any two positive coprime integers $a$ and $d$, there are infinitely many primes of the form $a + nd$, where $n$ is a non-negative integer.
\end{thm}

\begin{proof}
First, for the convenience, we just copy the figure from \cite{Xu-Zhang-Zhou} here.

\begin{figure}[htbp]
\centering
\begin{tabular}{|c|c|c|c|c|c|c|c|c|c|c|}
\hline
1& {\bf\color{red}\sout{7}} & 11 & 13 & 17 & 19 & 23 & 29 & 31 & {\it 37} & {\it 41} \\
\hline
 & {\it 43} & {\it 47} & {\bf\color{blue}\sout{49}} & {\it 53} & {\it 59} & {\it 61}& {\bf 67} & {\bf 71} & {\bf 73} & {\bf\color{blue}\sout{77}} \\
 & {\bf 79} & {\bf 83} & {\bf 89} & {\bf\color{blue}\sout{91}} & {\it 97} & {\it 101}&{\it 103} &{\it 107}&{\it 109} &{\it 113} \\
$M_7^{(0)}$ &{\bf\color{blue}\sout{119}} &{\it 121} &{\bf 127} &{\bf 131} &{\bf\color{blue}\sout{133}} &{\bf 137} &{\bf 139} &{\bf 143}&{\bf 149} & {\bf 151} \\
 &{\it 157} &{\bf\color{blue}\sout{161}} &{\it 163} &{\it 167} &{\it 169} &{\it 173} &{\it 179} &{\it 181}&{\bf 187} & {\bf 191} \\
 &{\bf 193} &{\bf 197} &{\bf 199} &{\bf\color{blue}\sout{203}} & {\bf 209} & {\bf 211}& & & & \\
\hline
\hline
 & {\bf\color{red}\sout{217}} & 221 & 223 & 227 & 229 & 233 & 239 & 241 & 247 & 251 \\
 & 253 & 257 & {\bf\color{blue}\sout{259}} & 263 & 269 & 271 & 277 & 281 & 283 & {\bf\color{blue}\sout{287}} \\
 & 289 & 293 & 299 & {\bf\color{blue}\sout{301}} & 307 & 311 & 313 & 317 & 319 & 323 \\
$M_7^{(1)}$ & {\bf\color{blue}\sout{329}} & 331 & 337 & 341 & {\bf\color{blue}\sout{343}} & 347 & 349 & 353 & 359 & 361 \\
 & 367 & {\bf\color{blue}\sout{371}} & 373 & 377 & 379 & 383 & 389 & 391 & 397 & 401 \\
 & 403 & 407 & 409 & {\bf\color{blue}\sout{413}} & 419 & 421 & & & & \\
\hline
\hline
 & {\bf\color{red}\sout{427}} & 431 & 433 & 437 & 439 & 443 & 449 & 451 & 457 & 461 \\
 & 463 & 467 & {\bf\color{blue}\sout{469}} & 473 & 479 & 481 & 487 & 491 & 493 & {\bf\color{blue}\sout{497}} \\
 & 499 & 503 & 509 & {\bf\color{blue}\sout{511}} & 517 & 521 & 523 & 527 & 529 & 533 \\
$M_7^{(2)}$ & {\bf\color{blue}\sout{539}} & 541 & 547 & 551 & {\bf\color{blue}\sout{553}} & 557 & 559 & 563 & 569 & 571 \\
 & 577 & {\bf\color{blue}\sout{581}} & 583 & 587 & 589 & 593 & 599 & 601 & 607 & 611 \\
 & 613 & 617 & 619 & {\bf\color{blue}\sout{623}} & 629 & 631 & & & & \\
\hline
\hline
 & $\cdots$ & & & & & & & & & \\
\end{tabular}
\caption{Obtain $M_{7}$ from $M_5$ by deleting the $7h$, $h\in M_5$}
\label{fig:M7}
\end{figure}

\begin{figure}[htbp]
\centering
\begin{tabular}{|c|c|c|c|c|c|c|c|c|c|}
\hline
1 & 11 & 13 & 17 & 19 & 23 & 29 & \ldots & 209 & 211 \\
\hline
 & 221 & 223 & 227 & 229 & 233 & 239 & \ldots & 419 & 421 \\
\hline
 & 431 & 433 & 437 & 439 & 443 & 449 & \ldots & 629 & 631 \\
\hline
 & $\cdots$ & & & & & & & & \\
\end{tabular}
\caption{$M_{7}$}
\label{fig:M7-new}
\end{figure}

We look at the columns of Figure \ref{fig:M7-new}, $L(\mathcal{P}_7)=7\cdot 5\cdot 3\cdot 2=210$. Thus, the columns consist the arithmetic progressions:
\[
x_{ik}=a_i+(k-1)\cdot 210,
\]
where $a_i\in M_7^{(0)}=\{11,13,17,19,23,29,31,37,41,43,47,\ldots,209,211\}$.

First, there exists at least one number $a_{i_0}\in M_7^{(0)}$ such that the set $\{x_{i_0 k}\mid x_{i_0 k}=a_{i_0}+(k-1)210\}$ contains infinitely many primes since there are infinitely many primes. Suppose that all $a_j\in M_7^{(0)}$ except $a_{i_0}$ such that each set $\{x_{jk}\mid x_{jk}=a_{j}+(k-1)210\}$ contains finite primes. That is for $k>K$ large enough, the number $x_{jk}=a_{j}+(k-1)210$ is always a composite number. Let $L_{a_{i_0}}$ denote the column or line where the set $\{x_{i_0 k}\mid x_{i_0 k}=a_{i_0}+(k-1)210\}$ is located in Figure \ref{fig:M7-new}. Choose a prime $p_m$ on the line $L_{a_{i_0}}$ satisfying $p_m>a_{i_0}+210(K-1)$. Then by Proposition \ref{prop:2}, the first several items in $M_{p_m}$ will not contain the gap (or skip) $2$.

But every pattern $\mathcal{P}_{p_m}$ has skip 2. In fact by Theorem \ref{thm:about-skips}, there are
\[
(3-2)(5-2)(7-2)\cdots(p_m-2)
\]
gaps of skip $2$ in $\mathcal{P}_{p_m}$. We can choose some of the corresponding pair of composite numbers in $\mathcal{P}_{p_m}$ which are located in the adjacent two columns denoted by $L_u$ and $L_v$. Here $u,v\in M_7^{(0)}$ and $v-u=2$.

Note that there are $15=(3-2)(5-2)(7-2)$ gaps of $2$ in $D_{7}$. And the same number of gap $4$ in $D_{7}$. Hence, we can assume that $L_u$ and $L_v$ are not the column passing through $a_{i_0}$.

Since every two odd composite numbers with gap $2$ are coprime, the elements at the same level of $L_u$ and $L_v$(i.e., they are the form like $u+210(k_s-1)$ and $v+210(k_s-1)$) have different prime divisors.

Hence, without loss of generality, we can choose $L_u$ and $L_v$ such that on them there are at least
\[
N:=\biggl[(3-2)(5-2)(7-2)\cdots(p_m-2)\cdot\frac{1}{15}\biggr]
\]
pairs of composite numbers with gap $2$ in $M_{p_m}^{(0)}$. Each pair of such composite numbers have at least four distinct prime divisors. In the Figure \ref{fig:main}, $*$ is a central position (not a number) in $M_{p_m}^{(0)}$. For the pairs of red points in the figure, they only can be eliminated by the primes in the set $\{p_{m+1},p_{m+2},\ldots,p_{m+k}\}$. Here $p_{m+k}$ is the largest prime less than $\sqrt{X/2}$, and $X=1+L(\mathcal{P}_{p_n})=1+\prod_{i=1}^{m}p_i$ is the last number in $M_{p_m}^{(0)}$.

There are at least $N-\frac{\sqrt{X/2}}{210}$ pair of red points with skip $2$ on the lines $L_u$ and $L_v$. They should be eliminated by these primes $p_{m+1},p_{m+2},\ldots,p_{m+k}$.

\begin{figure}[htbp]
  \centering
\begin{tikzpicture}[scale=1]

\draw[gray] (0,10) -- (0,0);
\draw[gray] (0,0) -- (5,0);
\draw[gray] (5,0) -- (5,10);
\draw[gray] (5,10) -- (0,10);

\draw[gray] (0,8) -- (5,8);
\draw[gray] (0,4) -- (5,4);
\draw[gray,dashed] (0,6) -- (5,6);

\draw[gray,dashed] (0,7.99) -- (5,7.99);
\draw[gray,dashed] (0,4.01) -- (5,4.01);
\draw[gray,dashed] (0.01,4.01) -- (0.01,7.99);
\draw[gray,dashed] (4.99,4.01) -- (4.99,7.99);

\draw[gray,dashed] (0,6.75) -- (5,6.75);

\draw[gray] (2,10) -- (2,0); 
\draw[gray] (3,10) -- (3,0); 
\draw[gray] (3.5,10) -- (3.5,0);     

\draw[gray] (3,10.4) -- (3,10.6);
\draw[gray] (3.5,10.4) -- (3.5,10.6);
\draw[gray][<->] (3,10.5) -- (3.5,10.5);
\node  at (3.25,10.7) {\footnotesize $2$};

\node  at (0,7.5)[anchor=east] {\footnotesize $M_{p_m}$};
\node  at (2,7.8)[anchor=east] {\footnotesize $p_{m+1}$};
\node  at (2,7.6)[anchor=east] {\footnotesize $p_{m+2}$};
\node  at (1.75,7.35)[anchor=east] {\footnotesize $\vdots$};
\node  at (2,6.9)[anchor=east] {\footnotesize $p_{m+k}$};
\node  at (2.5,6.9) {\footnotesize $<\sqrt{X/2}$};

\node  at (3,10.4)[anchor=north] {\footnotesize $L_u$};
\node  at (3.5,10.4)[anchor=north] {\footnotesize $L_v$};

\node  at (2.5,6) {\footnotesize $*$};
\node  at (5,4.2)[anchor=east] {\footnotesize $X$};

\fill [black] ($(0.5,9.8)$) circle (1.5pt);
\fill [black] ($(0.5,9.3)$) circle (1.5pt);
\fill [black] ($(0.5,8.6)$) circle (1.5pt);

\fill [black] ($(1,9.6)$) circle (1.5pt);
\fill [black] ($(1,8.9)$) circle (1.5pt);
\fill [black] ($(1,8.3)$) circle (1.5pt);
\fill [black] ($(1,8.6)$) circle (1.5pt);

\fill [black] ($(1.5,9.3)$) circle (1.5pt);
\fill [black] ($(1.5,8.8)$) circle (1.5pt);

\fill [black] ($(2.5,9.8)$) circle (1.5pt);
\fill [black] ($(2.5,8.9)$) circle (1.5pt);
\fill [black] ($(2.5,8.2)$) circle (1.5pt);

\fill [black] ($(3,9.5)$) circle (1.5pt);
\fill [black] ($(3,8.7)$) circle (1.5pt);
\fill [black] ($(3,8.4)$) circle (1.5pt);

\fill [black] ($(3.5,9.6)$) circle (1.5pt);
\fill [black] ($(3.5,9.3)$) circle (1.5pt);
\fill [black] ($(3.5,8.5)$) circle (1.5pt);

\fill [black] ($(4,9.4)$) circle (1.5pt);
\fill [black] ($(4,8.6)$) circle (1.5pt);
\fill [black] ($(4,8.2)$) circle (1.5pt);

\fill [black] ($(4.5,9.7)$) circle (1.5pt);
\fill [black] ($(4.5,9.2)$) circle (1.5pt);
\fill [black] ($(4.5,8.3)$) circle (1.5pt);
\fill [black] ($(2,9.8)$) circle (1.5pt);
\fill [black] ($(2,9.3)$) circle (1.5pt);
\fill [black] ($(2,8.8)$) circle (1.5pt);
\fill [black] ($(2,8.3)$) circle (1.5pt);
\fill [black] ($(2,7.8)$) circle (1.5pt);
\fill [black] ($(2,7.6)$) circle (1.5pt);
\fill [black] ($(2,7.2)$) circle (1.5pt);
\fill [black] ($(2,6.9)$) circle (1.5pt);
\fill [black] ($(2,6.6)$) circle (1.5pt);
\fill [black] ($(2,6.3)$) circle (1.5pt);
\fill [black] ($(2,5.8)$) circle (1.5pt);
\fill [black] ($(2,5.6)$) circle (1.5pt);
\fill [black] ($(2,5.3)$) circle (1.5pt);
\fill [black] ($(2,4.8)$) circle (1.5pt);
\fill [black] ($(2,4.6)$) circle (1.5pt);
\fill [black] ($(2,4.3)$) circle (1.5pt);
\fill [black] ($(2,3.8)$) circle (1.5pt);
\fill [black] ($(2,3.6)$) circle (1.5pt);
\fill [black] ($(2,3.3)$) circle (1.5pt);
\fill [black] ($(2,2.8)$) circle (1.5pt);
\fill [black] ($(2,2.6)$) circle (1.5pt);
\fill [black] ($(2,2.3)$) circle (1.5pt);
\fill [black] ($(2,1.8)$) circle (1.5pt);
\fill [black] ($(2,1.6)$) circle (1.5pt);
\fill [black] ($(2,1.3)$) circle (1.5pt);
\fill [black] ($(2,0.8)$) circle (1.5pt);
\fill [black] ($(2,0.6)$) circle (1.5pt);
\fill [black] ($(2,0.3)$) circle (1.5pt);

\fill [blue] ($(3,  7.8)$) circle (2pt);
\fill [blue] ($(3.5,7.8)$) circle (2pt);
\fill [blue] ($(3,  7.5)$) circle (2pt);
\fill [blue] ($(3.5,7.5)$) circle (2pt);
\fill [blue] ($(3,  7.1)$) circle (2pt);
\fill [blue] ($(3.5,7.1)$) circle (2pt);
\fill [red] ($(3,  6.5)$) circle (2pt);
\fill [red] ($(3.5,6.5)$) circle (2pt);
\fill [red] ($(3,  6.2)$) circle (2pt);
\fill [red] ($(3.5,6.2)$) circle (2pt);
\fill [red] ($(3,  5.6)$) circle (2pt);
\fill [red] ($(3.5,5.6)$) circle (2pt);
\fill [red] ($(3,  5.3)$) circle (2pt);
\fill [red] ($(3.5,5.3)$) circle (2pt);
\fill [red] ($(3,  5.0)$) circle (2pt);
\fill [red] ($(3.5,5.0)$) circle (2pt);
\fill [red] ($(3,  4.5)$) circle (2pt);
\fill [red] ($(3.5,4.5)$) circle (2pt);
\fill [red] ($(3,  4.2)$) circle (2pt);
\fill [red] ($(3.5,4.2)$) circle (2pt);

\end{tikzpicture}
  \caption{Here $*$ is the center of the $M_{p_m}^{(0)}$}\label{fig:main}
\end{figure}
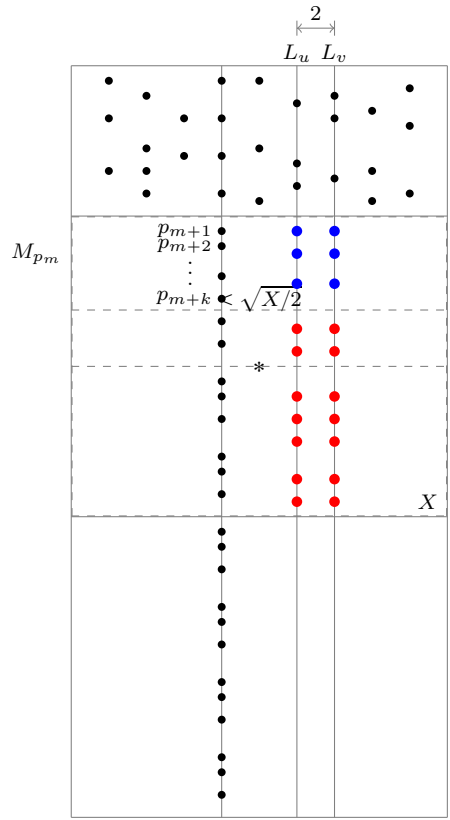

During the process of the elimination by the primes, we should take into account the composite numbers with same prime divisors. Note that the largest value of the red points is greater or equal than
\begin{equation}\label{eqn:main}
210\cdot\biggl(\frac{1}{15}\prod_{i=2}^{m}(p_i-2)-\frac{\sqrt{X/2}}{210}\biggr)+\sqrt{X/2}=14\prod_{i=2}^{m}(p_i-2).
\end{equation}


Thus we should have the following inequality
\begin{equation}\label{eqn:main-inequality}
14\prod_{i=2}^{m}(p_i-2)\cdot(1-\frac{1}{p_{m+1}})(1-\frac{1}{p_{m+2}})\cdots(1-\frac{1}{p_{m+k}})\leqslant\pi(X)-m.
\end{equation}
Note that the left hand side
\[
\begin{split}
14\prod_{i=2}^{m}(p_i-2)\cdot\prod_{i=m+1}^{m+k}(1-\frac{1}{p_{i}})
&=28\prod_{i=2}^{m}\frac{p_i-2}{1-\frac{1}{p_{i}}}\cdot\prod_{i=1}^{m+k}(1-\frac{1}{p_{i}})\\
&=28\cdot\prod_{i=2}^{m}p_i\cdot\prod_{i=2}^{m}(1-\frac{1}{p_i-1})\cdot\prod_{i=1}^{m+k}(1-\frac{1}{p_i})\\
&>14\cdot\prod_{i=2}^{m}p_i\cdot\prod_{i=2}^{m}(1-\frac{1}{p_i})\cdot\prod_{i=1}^{m+k}(1-\frac{1}{p_i})\\
&=14\cdot\prod_{i=1}^{m}p_i\cdot\prod_{i=1}^{m}(1-\frac{1}{p_i})\cdot\prod_{i=1}^{m+k}(1-\frac{1}{p_i})\\
\end{split}
\]
By Merten's estimates\cite{Hildebrand} in Lemma \ref{lem:Merten-estimate}, we have
\[
\begin{split}
&14\cdot\prod_{i=1}^{m}p_i\cdot\prod_{i=1}^{m}(1-\frac{1}{p_i})\cdot\prod_{i=1}^{m+k}(1-\frac{1}{p_i})\\
=&14(X-1)\cdot\frac{e^{-\gamma}}{\log p_m}\Bigl(1+O(\frac{1}{\log p_m})\Bigr)\cdot\frac{e^{-\gamma}}{\log p_{m+k}}\Bigl(1+O(\frac{1}{\log p_{m+k}})\Bigr)\\
>&13X\cdot\frac{e^{-\gamma}}{\log p_m}\cdot\frac{e^{-\gamma}}{\log p_{m+k}}\\
>&\frac{13}{e^{2\gamma}\log p_m}\cdot\frac{X}{\log\sqrt{X/2}}\\
>&\frac{26}{e^{2\gamma}\log p_m}\cdot\frac{X}{\log X}.
\end{split}
\]
It will violate the inequality \eqref{eqn:main-inequality} if the prime number $p_m$ satisfies $\frac{26}{e^{2\gamma}\log p_m}\geqslant 2\log 2$ due to the inequality \eqref{eqn:Erdos}.

 Note that $e^{\gamma}\approx 1.78107241799$ \cite{Euler-constant}. Let $p_m\leqslant e^{5.9}\approx 365.04$, then we will have
\[
\frac{26}{e^{2\gamma}\log p_m}> 2\log 2.
\]

  If $p_m>e^{5.9}$, we change $M_7$ to some $M_{p_n}$ with longer period. Then the coefficient in \eqref{eqn:main} will be $\prod_{i=1}^{n}p_i/\prod_{i=2}^{n}(p_i-2)$.

Next we follow the same arguments. In detail, equation \eqref{eqn:main} becomes
\[
\prod_{i=1}^{n}p_i\cdot\biggl(\frac{1}{\prod_{i=2}^{n}(p_i-2)}\prod_{i=2}^{m}(p_i-2)-\frac{\sqrt{X/2}}{\prod_{i=1}^{n}p_i}\biggr)+\sqrt{X/2}=\frac{\prod_{i=1}^{n}p_i}{\prod_{i=2}^{n}(p_i-2)}\cdot\prod_{i=2}^{m}(p_i-2).
\]
And \eqref{eqn:main-inequality} becomes
\begin{equation}\label{eqn:main-inequality2}
\frac{\prod_{i=1}^{n}p_i}{\prod_{i=2}^{n}(p_i-2)}\cdot\prod_{i=2}^{m}(p_i-2)\cdot(1-\frac{1}{p_{m+1}})(1-\frac{1}{p_{m+2}})\cdots(1-\frac{1}{p_{m+k}})\leqslant\pi(X)-m.
\end{equation}

The left hand side becomes
\[
\begin{split}
&\frac{\prod_{i=1}^{n}p_i}{\prod_{i=2}^{n}(p_i-2)}\cdot\prod_{i=2}^{m}(p_i-2)\cdot\prod_{i=m+1}^{m+k}(1-\frac{1}{p_{i}})\\
=&2\cdot\prod_{i=2}^{n}\frac{p_i}{p_i-2}\cdot\prod_{i=2}^{m}\frac{p_i-2}{1-\frac{1}{p_{i}}}\cdot\prod_{i=2}^{m+k}(1-\frac{1}{p_{i}})\\
=&4\cdot\prod_{i=2}^{n}\frac{p_i}{p_i-2}\cdot\prod_{i=2}^{m}p_i\cdot\prod_{i=2}^{m}(1-\frac{1}{p_i-1})\cdot\prod_{i=1}^{m+k}(1-\frac{1}{p_i})\\
>&2\cdot\prod_{i=2}^{n}\frac{p_i}{p_i-2}\cdot\prod_{i=2}^{m}p_i\cdot\prod_{i=2}^{m}(1-\frac{1}{p_i})\cdot\prod_{i=1}^{m+k}(1-\frac{1}{p_i})\\
=&2\cdot\prod_{i=2}^{n}\frac{p_i}{p_i-2}\cdot\prod_{i=1}^{m}p_i\cdot\prod_{i=1}^{m}(1-\frac{1}{p_i})\cdot\prod_{i=1}^{m+k}(1-\frac{1}{p_i}).\\
\end{split}
\]
By Merten's estimates, we have
\[
\begin{split}
&2\cdot\prod_{i=2}^{n}\frac{p_i}{p_i-2}\cdot\prod_{i=1}^{m}p_i\cdot\prod_{i=1}^{m}(1-\frac{1}{p_i})\cdot\prod_{i=1}^{m+k}(1-\frac{1}{p_i})\\
=&2\cdot\prod_{i=2}^{n}\frac{p_i}{p_i-2}\cdot(X-1)\cdot\frac{e^{-\gamma}}{\log p_m}\cdot\biggl(1+O(\frac{1}{\log p_m})\biggr)\cdot\frac{e^{-\gamma}}{\log p_{m+k}}\cdot\biggl(1+O(\frac{1}{\log p_{m+k}})\biggr)\\
>&(2\prod_{i=2}^{n}\frac{p_i}{p_i-2}-1)X\cdot\frac{e^{-2\gamma}}{\log p_m\cdot\log p_{m+k}}\\
>&(2\prod_{i=2}^{n}\frac{p_i}{p_i-2}-1)X\cdot\frac{1}{e^{2\gamma}\log p_m}\cdot\frac{1}{\log\sqrt{X/2}}\\
>&\frac{2(2\prod_{i=2}^{n}\frac{p_i}{p_i-2}-1)}{e^{2\gamma}\log p_m}\cdot\frac{X}{\log X}=:\alpha\cdot\frac{X}{\log X}.
\end{split}
\]

\noindent{\bf Claim.} For this $p_m$, there exists a number $n$ with $p_n < p_m$, such that
\[
2(2\prod_{i=2}^{n}\frac{p_i}{p_i-2}-1)> (2\log 2)\cdot e^{2\gamma}\log p_m.
\]
That is, $\alpha > 2\log 2$.

Hence, if denote the left hand side of \eqref{eqn:main-inequality2} as $LHS$, then
\[
LHS>(2\log 2)\frac{X}{\log X}>\pi(X),
\]
which violates the inequality \eqref{eqn:main-inequality2}.

Thus our assumption is wrong. Note that here we work on the table of $M_{p_n}$ similar with Figure \ref{fig:M7-new} and Figure \ref{fig:main}. Hence, there exists at least another $a_{j_0}$ such that the set $\{x_{j_0 k}\mid x_{j_0 k}=a_{j_0}+(k-1)\prod_{i=1}^{n}p_i\}$ contains infinitely many primes.

{\bf In fact} by above argument we have proved that {\bf each line of the pair} $L_u$ and $L_v$ with gap $2$ or $4$ in the table about $M_{p_n}$ contains infinitely many primes.

Obviously, it infers that {\bf each line of the pair} $L_u$ and $L_v$ with gap $2$ or $4$ in the table about $M_{7}$ (see Figure \ref{fig:main}) contains infinitely many primes. And it also true that for the table of $M_3$ and $M_5$. Therefore, since the bigger gap such as $6$ comes from add $2$ and $4$, we conclude that the pair lines $L_u$ and $L_v$ with gap $6$ also both contains infinitely many primes.

Therefore, every column $\{x_{j k}\mid x_{j k}=a_{j}+(k-1)\prod_{i=1}^{n}p_i\}$ (the $j$-th column) contains infinitely many primes. Here $n=1,2,3,\ldots$.

By using the same idea, we can prove that if fix one of rows of the block $M_7$ in Figure \ref{fig:M7-new}, then the set
\begin{equation}\label{eqn:2}
\{x_{in}\mid x_{in}=a_{i}+n(k_0-1)\cdot 210,\quad n=1,2,3,\ldots\}
\end{equation}
also contains infinitely many primes.

For the general arithmetic progression $A:=\{ak+d\mid k\in\mathbb{N}\}$, $(a,d)=1$. First we can choose suitable $p_k$ such that $d$ or $ah+d$ belongs the set $M_{p_k}$ for some $h$. Then fix this $p_k$, considering the similar rectangle as Figure \ref{fig:M7-new}. We consider the subset
\[
\{a(\prod_{i=1}^{n}p_i)k+d\mid k=1,2,\ldots\}.
\]
Thus it contains infinitely many primes. It completes the proof of the theorem of Dirichlet.
\end{proof}

\bigskip

\section{Estimation of the prime counting function in special cases}

In $M_{p_m}^{(0)}$, there are $T_{p_m}$ elements in it. Then after deleting the numbers of the form $p_ih$, $i=m+1,m+2,\ldots,m+K$, the left are all prime numbers. Here we assume $p_{m+K}$ is the largest prime number which is less than or equal to $X:=L(\mathcal{P}_{p_m})=1+\prod_{i=1}^{m}p_i$. Hence we have
\[
\begin{split}
\pi(X)-m&\geqslant T_{p_m}\cdot(1-\frac{1}{p_{m+1}})\cdot(1-\frac{1}{p_{m+2}})\cdots(1-\frac{1}{p_{m+K}})\\
&=\prod_{i=1}^{m}(p_i-1)\cdot\frac{\prod_{i=1}^{m+K}(1-p_{i}^{-1})}{\prod_{i=1}^{m}(1-p_{i}^{-1})}\\
&=\prod_{i=1}^{m}p_i\cdot\prod_{i=1}^{m+K}(1-p_{i}^{-1})
\end{split}
\]

Recall the Euler's product formula for zeta function
\[
\zeta(s)=\sum_{n=1}^{\infty}\frac{1}{n^s}=\prod_{p\ \mathrm{prime}}\frac{1}{1-p^{-s}}.
\]
For the case $s=1$, we will use the following lemma.
\begin{lem}[Merten's estimates\cite{Hildebrand}]\label{lem:Merten-estimate}
\[
\prod_{p\leqslant x}(1-\frac{1}{p})=\frac{e^{-\gamma}}{\log x}\biggl(1+O(\frac{1}{\log x})\biggr),
\]
where $\gamma$ is Euler's constant.
\end{lem}

\begin{lem}
For any positive integer $N$, we have the inequality
\[
\log(1+N)<\sum_{n=1}^{N}\frac{1}{n}<1+\log N.
\]
\end{lem}

Hence, by the above lemmas, we have
\[
\begin{split}
\pi(X)&\geqslant\prod_{i=1}^{m}p_i\cdot\prod_{i=1}^{m+K}(1-p_{i}^{-1})+m\\
&=(X-1)\cdot\frac{e^{-\gamma}}{\log X}\biggl(1+O(\frac{1}{\log X})\biggr)+m.
\end{split}
\]
On the other hand, it is easy to see that
\[
\begin{split}
\pi(X)&\leqslant T_{p_m}\cdot 2\cdot(1-\frac{1}{p_{m+1}})\cdot(1-\frac{1}{p_{m+2}})\cdots(1-\frac{1}{p_{m+K}})\\
&=2\prod_{i=1}^{m}p_i\cdot\prod_{i=1}^{m+K}(1-p_{i}^{-1})+m\\
&=2(X-1)\cdot\frac{e^{-\gamma}}{\log X}\biggl(1+O(\frac{1}{\log X})\biggr)+m.
\end{split}
\]
Therefore, we have
\[
\frac{1}{e^{\gamma}}\cdot\frac{X-1}{\log X}\biggl(1+O(\frac{1}{\log X})\biggr)+m
\leqslant\pi(X)\leqslant
\frac{2}{e^{\gamma}}\cdot\frac{X-1}{\log X}\biggl(1+O(\frac{1}{\log X})\biggr)+m.
\]
Since $e^m<1+p_1p_2\cdots p_m=X$, $m<\log X$. Thus we have
\[
\pi(X)\asymp\frac{X}{\log X},\quad\text{where}\ X=1+p_1p_2\cdots p_m.
\]
Therefore, we get an inequality of prime counting function like the Chebyshev's inequality.


\bigskip

\noindent{\bf Acknowledgments :}
We would like to express our gratitude to Professor Vilmos Komornik who invite the author to visit the Department of Mathematics in University of Strasbourg. This work is done during the stay of the author in Strasbourg.


\bigskip

\noindent Haifeng Xu\\
School of Mathematical Sciences\\
Yangzhou University\\
Jiangsu China 225002\\
hfxu@yzu.edu.cn\\
\medskip
%
%


\end{document}